\documentclass[reqno]{amsart}
\usepackage{amsmath, amssymb, amsthm, cite}
\usepackage{comment}
\usepackage{xcolor}
\newtheorem{theorem}{Theorem}[section]
\newtheorem{lemma}[theorem]{Lemma}
\newtheorem{lem}[theorem]{Lemma}

\newtheorem{prop}[theorem]{Proposition}
\newtheorem{corollary}[theorem]{Corollary}

\newtheorem{remark}[theorem]{Remark}

\newtheorem{ques}[theorem]{Question}

\newcommand{\cM}{{\mathcal M}}
\newcommand{\cN}{{\mathcal N}}
\newcommand{\cA}{{\mathcal A}}
\newcommand{\cB}{{\mathcal B}}

\newcommand{\bR}{{\mathbb R}}

\newcommand{\abs}[1]{\left| #1 \right|}
\newcommand{\tnorm}[1]{\left\vert\kern-0.25ex\left\vert\kern-0.25ex\left\vert #1 
	\right\vert\kern-0.25ex\right\vert\kern-0.25ex\right\vert}

\newcommand{\norm}[1]{\left\lVert#1\right\rVert}

\usepackage{hyperref}					
\hypersetup{colorlinks,
	linkcolor=blue,%
	citecolor=blue}
\begin{document}

\author[J. Huang]{Jinghao Huang}
\address{Institute for Advances Study in Mathematics, Harbin Institute of Technologies,
Harbin, 150001, China}
\email{jinghao.huang@hit.edu.cn}

\author[K. Kudaybergenov]{Karimbergen Kudaybergenov}
\address{Institute for Advances Study in Mathematics, Harbin Institute of Technologies,
Harbin, 150001, China and
Suzhou Research Institute of Harbin Institute of Technology, Suzhou, 215104, China}
\email{kudaybergenovkk@gmail.com}

\author[B. Yan]{Bing Yan$^{\ast}$}\thanks{*The corresponding author.}
\address{Institute for Advances Study in Mathematics, Harbin Institute of Technologies,
Harbin, 150001, China}
\email{bingyan202510@163.com}

\title[Isometries on algebras of locally measurable operators]{Structure of linear isometries on algebras of locally measurable operators}



\begin{abstract}
Let $LS(\cM)$ be the algebra  of locally measurable
operators affiliated with a von Neumann algebra $\cM$, equipped with an
$F$-norm defined via a dimension function and a probability measure. 
We prove that every bijective linear isometry between $LS(\cM)$ admits a canonical representation of the form $\Phi(x)=wJ(x)$,
where $w$ is a unitary element and $J$ is a Jordan $^*$-isomorphism, which   extends classical results such as the Banach--Stone theorem and
Kadison's theorem.
Under several structural
assumptions on the underlying von Neumann algebras (including all type
$\mathrm{II}_\infty$ and  type $\mathrm{III}$ algebras, and all  factors, and  algebras with atomless
centers), we prove the one-to-one correspondence between 
 the $F$-norm and   the pair $(\mu, D)$  of a probability measure and a dimension function, 
which fails for algebras 
 with
atomic centers.
\end{abstract}

\maketitle

\section{Introduction}

An \textit{isometry} between operator algebras equipped with metrics is a (not necessarily surjective) linear map $\Phi:\mathcal{A}\to\mathcal{B}$ satisfying\cite{Ro}
\begin{equation*}
d_\cB(\Phi(x),\Phi(y) )=d_\cA(x,y )\quad \text{for all } x,y\in\mathcal{A}. 
\end{equation*}
A seminal result  is due to Kadison~\cite{K51}, who established a noncommutative version of the classical Banach--Stone Theorem\cite{K51}:
\begin{quote}
Every   linear surjective isometry  $\Phi:\mathcal{A}\to\mathcal{B}$ between unital $C^*$-algebras admits the following form
\begin{align}\label{descriptionofKadison}
  \Phi(x)=u\,J(x),\; x\in\mathcal{A},
\end{align}
where $u\in\mathcal{B}$ is a unitary   and $J:\mathcal{A}\to\mathcal{B}$ is a Jordan $^*$-isomorphism.
\end{quote}
Russo and Dye \cite{RD66} generalized Kadison’s theorem, proving that any linear map between unital $C^*$-algebras which sends unitaries to unitaries can be expressed as the form \eqref{descriptionofKadison}. 
Throughout this paper, we only consider linear isometries, and therefore, the term ``linear'' may be omitted. 

Recently, isometries on  the $L_0$-spaces and the Haagerup--Schultz algebras have been thoroughly studied in \cite{BHS24} (see also \cite{HKS} for $L_0$-isometries with respect to another metric) and \cite{HKSY}, respectively. 
For   descriptions of isometries on  commutative or  noncommutative $L_p$-spaces and, more generally,  symmetric spaces, we refer the reader to  \cite{Y81,HSZ20, HS24} and references therein.


Let $\cM$ be a von Neumann algebra with the center $Z(\cM)\cong L_\infty(\Omega, \Sigma, \mu)$, where $\mu$ is a strictly positive probability measure.
Denote by $L^+(\Omega, \Sigma, \mu)$ the set of all measurable functions defined on $(\Omega, \Sigma, \mu)$ and taking values in the extended half-line $[0,\infty]$ (functions that are equal almost everywhere are identified).
Let $D:P(\cM) \to L^+(\Omega, \Sigma, \mu)$ be a  dimension function.
For a function $f\in L^+(\Omega, \Sigma, \mu)$ and $\lambda\in\mathbb{R}$, we use the notation $[f>\lambda]$ to denote the superlevel set $\{\omega\in \Omega: f(\omega)>\lambda\}$.

Define the $F$-norm on $LS(\cM)$, the algebra of all locally measurable operators with respect to $\cM$, as \cite[page 262]{Yeadon 1973}  (see also in \cite[Section 4.3]{BCLSZ}) 
\begin{align}\label{LSM-norm}
 \|x\|_{\mu,D}=\inf\limits_{\lambda>0}\max\{\lambda,\mu([D(e_{(\lambda,\infty)}(\abs{x}))>\lambda])\}.
\end{align}
where $e_{(\lambda,\infty)}(|x|)$ is the spectral projection of $|x|$ corresponding to the interval $(\lambda, \infty)$.
In general, the $F$-norm on $LS(\cM)$ defined by \eqref{LSM-norm} depends on $\mu$ and $D$. To emphasize this dependence, we denote the $F$-norm by $\left\|\cdot\right\|_{\mu,D}$. 
Notably, the $F$-norm was introduced by Yeadon \cite{Yeadon 1973} to study the local measure topology (see \cite[Theorem 3.3]{Yeadon 1973} and \cite[Chapter 4.3]{BCLSZ}).

Note that the local measure topology on $LS(\cM)$ is metrizable if and only if the center $Z(\cM)$ of $\cM$ is countable decomposable \cite[Proposition 6.2]{MC}. 
Throughout this paper, unless otherwise stated, we only consider von Neumann algebras with countable decomposable centers.  

Let $(\Omega_i, \Sigma_i, \mu_i)$, $i=1,2$ be two probability measure spaces such that 
$Z(\cM)\cong L_\infty(\Omega_i, \Sigma_i, \mu_i)$, $i=1,2$. 
Let $D_i: P(\cM)\to L^+(\Omega_i, \Sigma_i, \mu_i)$, $i=1,2$ be  two dimension functions on $P(\cM)$.

We say that the pairs $(\mu_1, D_1)$ and $(\mu_2,D_2)$ are \emph{equivalent} if
\[
\left\|\cdot\right\|_{\mu_1,D_1}=\left\|\cdot\right\|_{\mu_2,D_2}.
\]

Since $LS(\cM)$ is  
the largest possible bimodule  affiliated with $\cM$ \cite{Y74}, it is of considerable interest to investigate the structure of the surjective isometries defined on them. 
Our main objective is to characterize  bijective linear isometries between such algebras and to understand 
the relationship between 
 the underlying von Neumann algebras and  these mappings. In particular, we show that every  bijective linear isometry of
$LS(\cM)$ with respect to the
$F$-norm $\left\|\cdot\right\|_{\mu,D}$ has a canonical form determined by a unitary element and a Jordan
$^*$-isomorphism between the algebras of locally measurable operators. 
Moreover, we study the relationships between pairs consisting of dimension functions and probability measures and the $F$-norm induced by these two quantities. Theorems \ref{local-isometry},  \ref{con-equiv}, and  \ref{case of type I} below provide a complete description of such isometries and clarify when the associated pairs consisting of a dimension function and a probability measure coincide.

\begin{theorem}\label{local-isometry} (see Section \ref{s4} for a complete proof)
Let $\mathcal{M}$ and $\mathcal{N}$ be von Neumann algebras, and let $(\mu_\cM,D_\cM)$ and $(\mu_\cN,D_\cN)$ be the corresponding pairs consisting of a probability measure and a dimension function. Suppose that $\Phi:LS(\cM) \to LS(\cN)$ is a linear bijection.
Then the following conditions are equivalent:
\begin{enumerate}
    \item $\Phi$ is an isometry from $\left(LS(\cM), \left\|\cdot\right\|_{\mu_\cM, D_\cM}\right)$ onto $\left(LS(\cN), \left\|\cdot\right\|_{\mu_\cN, D_\cN}\right)$;
    \item $\Phi$ admits a representation of the form
    \begin{align}\label{gen-form}
        \Phi(x) = w J(x), \quad x \in LS(\cM),
    \end{align}
    where $w$ is a unitary element in $\cN$ and $J:LS(\cM) \to LS(\cN)$ is a Jordan $^*$-isomorphism such that the pairs $\left(\mu_\cM, D_\cM\right)$ and $\left(\mu_\mathcal{N}, D_\mathcal{N}\circ J|_{P(\cM)}\right)$ are equivalent.
\end{enumerate}
\end{theorem}


In Section~\ref{s4}, we show that the  study of equivalence of pairs of probability measures and dimension functions reduced to the case when both measures are defined on same domain. Consequently, we come to the following question.

\begin{ques} Let $(\Omega, \Sigma, \mu)$ and $(\Omega, \Sigma, \mu')$ be probability measure spaces and let $D,D'$ be dimension functions on the lattice of projections $P(\cM)$ of the von Neumann algebra $\cM$ with values in $L^+(\Omega, \Sigma, \mu)$ and $L^+(\Omega, \Sigma, \mu')$, respectively. 
If the pairs $(\mu,D)$ and $(\mu',D')$ are equivalent, must we have $\mu=\mu'$ and $D=D'$?
\end{ques}

In the following result, we establish conditions on the von Neumann algebra under which equivalent pairs of probability measures and dimension functions coincide.

\begin{theorem}\label{con-equiv}
Let $\mathcal{M}$ be a von Neumann algebra  and let $(\mu,D)$ and $(\mu',D')$ be pairs consisting of a probability measure and a dimension function. Assume that the algebra $\mathcal{M}$ satisfies  one of the following conditions:
\begin{enumerate}
    \item[(a)] it is  of type II$_\infty$ or type III;
    \item[(b)] it is  of type II$_1$ with atomless center or it is a II$_1$ factor.
\end{enumerate}
If the pairs $(\mu, D)$ and $(\mu', D')$ are equivalent, then they coincide.
\end{theorem}

Let $\cM$ be a type I von Neumann algebra with an atomless center $Z(\cM)$.
Fix  a faithful normal semifinite extended center-valued trace $T$ on $\cM$ such that $T(p)=1$ on $\Omega$, where $p$ is a faithful ($c(p)=\mathbf{1}$) abelian projection in $\cM$.

Let $D,D'$ be two dimension functions on $P(\cM)$ with  $D(p)=c$, $D'(p)=c'$.

Set
\begin{align*}
c_{\min} =\min\{c, \|\mathbf{1}\|_{\mu,D}\},\,\,\, D_{\min}=c_{\min}T,
\end{align*}
\begin{align*}
c_{\min}' =\min\{c', \|\mathbf{1}\|_{\mu',D'}\},\,\,\, D_{\min}'=c_{\min}'T.
\end{align*}

The next result provides  a criterion  for the equivalence of pairs of probability measures and dimension functions for type I von Neumann algebras.

\begin{theorem}\label{case of type I}
Let $\mathcal{M}$ be a type I von Neumann algebra, and let $(\mu, D)$ and $(\mu', D')$ be pairs, each consisting of a probability measure and a dimension function. Assume that $\mathcal{M}$ either has an atomless center  or is a factor. Then $(\mu, D)$ and $(\mu', D')$ are equivalent if and only if $\mu = \mu'$ and $D_{\min} = D'_{\min}$.
\end{theorem}

Note that 
Theorems \ref{con-equiv} and \ref{case of type I} do not hold in the setting 
 when the von Neumann algebra has an atomic center, see Section \ref{s6}.

\begin{corollary}\label{uniq} 
Let $\mathcal{M}$ be a von Neumann algebra  and let $(\mu,D)$  be a pair consisting of  a probability measure and a dimension function. Assume that the algebra $\mathcal{M}$ satisfies  one of the following conditions: 
\begin{enumerate}
    \item[(a)] it is  of type II$_\infty$ or type III;
    \item[(b)] it is  of type II$_1$ with atomless center or it is a II$_1$ factor.
\end{enumerate}
Then the $F$-norm $\left\|\cdot \right\|_{\mu,D}$ is uniquely determined by its values on $P(\cM)$. 
\end{corollary}

Corollary~\ref{uniq} shows that, under the given conditions, the $F$-norm $\left\|\cdot\right\|_{\mu,D}$ is uniquely determined by its values on the projection lattice $P(\cM)$. 
This complements the known fact that a von Neumann algebra as a Jordan $^*$-algebra is uniquely determined by $P(\cM)$ \cite{AK2021, Mori2020, Mori2023}. 
Since the $F$-norm is constructed from a probability measure $\mu$ and a dimension function $D$, 
the corollary indicates that these data are essentially encoded in the projection lattice whenever the algebra is either sufficiently non-atomic or a continuous factor.
The cases covered exclude type I algebras and type II$_1$ algebras whose centers contain at least two atoms; for those, the projection lattice may not capture the fine structure needed to recover the $F$-norm uniquely. 
Hence, Corollary~\ref{uniq}   provides a sharp uniqueness criterion for the $F$-norm in terms of projections.

\section{Preliminaries}

\subsection{Locally measurable operators}

Let $H$ be a complex Hilbert space and let $B(H)$ denote the $^\ast$-algebra of all bounded linear operators on $H$. Let $\cM$ be a von Neumann algebra contained in $B(H)$. As usual, we denote by $P(\cM)$ the set of all projections in $\cM$.

Recall that a densely defined closed linear operator $x : \mathrm{dom}(x) \to H$ (where $\mathrm{dom}(x)$ is a dense linear subspace of $H$) is said to be \textit{affiliated} with $\cM$ if
\[
yx \subset xy \quad \text{for all } y \in \cM',
\]
where $\cM'$ denotes the commutant of $\cM$ (see \cite{BCLSZ,DPS}).
An operator $x$ affiliated with $\cM$ is called \textit{measurable} (with respect to $\cM$) if
\[
e_{(\lambda,\infty)}(|x|) \text{ is a finite projection for some } \lambda > 0,
\]
where $e_{(\lambda,\infty)}(|x|)$ is the spectral projection of $|x|$ corresponding to the interval $(\lambda, \infty)$.

If $\cM$ is semifinite and  $\tau$ is a faithful normal semifinite trace on $\cM$, a measurable  operator $x$ affiliated with $\cM$ is called \textit{$\tau$-measurable} if
\[
\tau\big(e_{(\lambda,\infty)}(|x|)\big) < \infty \text{ for some } \lambda > 0.
\]
We denote by $S(\cM)$ and $S(\cM, \tau)$ the sets of all measurable and $\tau$-measurable operators, respectively  (see \cite{BCLSZ,DPS}).

A closed linear operator $x$ acting on $H$ is said to be
locally measurable with respect to the von Neumann algebra $\cM$ if $x\eta \cM$ and there exists
a sequence of central projections $\{z_n\}_{n\ge1}$ in $P(Z(\cM))$ such that $z_n \uparrow\mathbf{1}$ and $xz_n \in  S(\cM)$
for all $n\ge1$.

The set of all operators locally measurable with respect to the von Neumann algebra $\cM$ is denoted by $LS(\cM)$.

For $x$, $y \in LS(\cM)$, it is well known that $x+y$ and $xy$ are densely defined and preclosed operators. Moreover, their closures, as well as $x^\ast$, belong to $LS(\cM)$.
Equipped with these operations, each of $LS(\cM)$, $S(\cM)$ and $S(\cM, \tau)$ forms a unital $^\ast$-algebra over $\mathbb{C}$ (see \cite{BCLSZ,DPS}).

Clearly, $\cM$ itself is a $^\ast$-subalgebra of $S(\cM, \tau)$, $S(\cM)$ and  $LS(\cM)$.

For an arbitrary $x \in LS(\cM)$, we denote:
\begin{itemize}
    \item \textit{left support} $l(x)$: the smallest projection $p \in \cM$ such that $px = x$;
    \item \textit{right support} $r(x)$: the smallest projection $q \in \cM$ such that $xq = x$;
    \item \textit{support projection} $s(x) = l(x) \vee r(x)$;
    \item \textit{central support} $c(x)$: the smallest central projection $z \in P(Z(\cM))$ such that $xz = x$.
\end{itemize}
Note that $s(|x|) = l(|x|) = r(|x|)$ for all $x \in LS(\cM)$.
A projection $p\in \cM$ is said to be faithful if $c(p)=\mathbf{1}$.

\subsection{Dimension functions on the lattice of projections}
Let $\mathcal{M}$ be an arbitrary von Neumann algebra. Assume  $\varphi$ is a $^*$-isomorphism from $Z(\cM)$ onto the $^*$-algebra $L_{\infty}(\Omega,\Sigma,\mu)$, where $(\Omega,\Sigma,\mu)$ is a probability space. Denote by $L^{+}(\Omega,\Sigma,\mu)$ the set of all measurable functions defined on $(\Omega,\Sigma,\mu)$ and taking values in the extended half-line $[0,\infty]$ (functions that are equal almost everywhere are identified). For $f\in L^{+}(\Omega,\Sigma,\mu)$, the support projection $s(f)$ is defined as (the equivalence class) $\chi_A$, where $A=\{w\in\Omega:f(w)\ne0\}\in\Sigma$.

A mapping $D:P(\cM)\to L^{+}(\Omega,\Sigma,\mu)$ is said to be a \textit{dimension function} if the following hold\cite[Definition~1.4]{Segal1953}:
\begin{enumerate}
\item If $p\in P(\cM)$ and $p\ne 0$, then $D(p)\ne0$;
\item \label{D-finite}The function $D(p)$ is almost everywhere finite if and only if $p$ is a finite projection;
\item \label{perp}$D(p\vee q)=D(p)+D(q)$ if $pq=0$;
\item $D(p)=D(q)$ if $p\sim q$, $p$, $q\in P(\cM)$;
\item \label{D-finite5}  $D(zp)=\varphi(z)D(p)$ for any $z\in P(Z(\cM))$ and $p\in P(\cM)$;
\item If $p_i$, $p\in P(\cM)$, $i\in I$, and $p_i\uparrow p$, then $D(p)=\sup\limits_{i\in I}D(p_i)$.
\end{enumerate}

Let $p$ be a faithful projection. Then $Z(p\mathcal{M}p)$ can be identified with $Z(\mathcal{M})p$ via the correspondence $x \in Z(\mathcal{M}) \mapsto xp \in p\mathcal{M}p$ \cite[Proposition II.3.10]{Tak1}. In this case, the restriction $D|_{P(p\mathcal{M}p)}$ defines a dimension function on $P(p\mathcal{M}p)$ (see the proof of \cite[Lemma 4.3.29]{BCLSZ}).

Since $\varphi:Z(\cM)\to L_{\infty}(\Omega,\Sigma,\mu)$ is  a $^*$-isomorphism, any projection $z\in P(Z(\cM))$ is in the form
\begin{align}\label{zA}
z= z_A = \varphi^{-1}(\chi_A),\,\, A\in \Sigma.
\end{align}

Let $\cM$ be a von Neumann algebra and let $z\in P(Z(\cM))$ be such that $\cM z$ is a semifinte von Neumann algebra and $\cM(\mathbf{1}-z)$ is a type III von Neumann algebra. Let   $T$ be a fixed faithful normal semifinite extended center-valued trace $T$ on  $\cM z$.
Any dimension  function   
$D:P(\cM)\to L^{+}(\Omega,\Sigma,\mu)$ is of the following  form~\cite[Theorems 1.11.2 and 1.11.13]{BCLSZ}
\begin{align}\label{dim-fc}
D(p)=c T(zp)+\infty\cdot\varphi(c((\mathbf{1}-z)p)),\,\, p\in P(\cM),
\end{align} 
for some strictly positive measurable function $c:\Omega\to \mathbb{R}_+$. In particular, if  $\cM$ is a semifinite von Neumann algebra, then $D(p)=cT(p)$ for any $p\in P(\cM)$, showing that for a semifinite von Neumann algebra the dimension function coincides with the reduction of the extended center-valued trace $T$ to $P(\cM)$ (see also \cite[Theorem 2.34]{Tak1}).

 For any properly infinite projection $p\in P(\cM)$, we have\cite[Theorem 1.11.3]{BCLSZ}
$$
D(p)=\infty\cdot\varphi(c(p)).
$$ 

 Assume  $\cM=B(H)$ is a type I-factor with the canonical trace ${\rm Tr}$ such that 
\[
{\rm Tr}(p)={\rm rank}(p), \,\, p\in P(\cM).
\]
Then the dimension function $D$ in \eqref{dim-fc} can be represented as follows
\[
D(p) = \delta \cdot {\rm rank}(p),\,\,p\in P(B(H)),
\]
where $\delta$ is a fixed positive number.

\subsection{$F$-norm on the algebra of locally measurable operators}
Let $\cM$ be a von Neumann algebra, for any $x\in LS(\cM)$, define\footnote{Since $\mu$ is a probability measure, the infimum on the right-hand side of \eqref{LSM-norm} is attained at some point in the interval $(0,1)$.}
\begin{align}\label{second-form}
\|x\|_{\mu,D} = \inf_{0<\lambda<1} \max\left\{ \lambda, \mu\!\left(\left[D\left(e_{(\lambda,\infty)}(|x|)\right) > \lambda\right]\right) \right\}.
\end{align}
where $e_{(\lambda,\infty)}(|x|)$ is the spectral projection of $|x|$ corresponding to the interval $(\lambda, \infty)$. 
As mentioned in the introduction, the mapping $\left\|\cdot\right\|_{\mu,D}:LS(\cM)\to \bR_+$ is an $F$-norm on $LS(\cM)$, 
and  the metric topology generated by the above $F$-norm coincides with the so-called local measure topology on $LS(\cM)$
(see \cite[Page 262]{Yeadon 1973}).

Note that \cite[Page 204]{BCLSZ}, the $F$-norm can also be represented in the following.
\begin{align}\label{first-form}
\|x\|_{\mu,D}=\inf\limits_{\lambda>0}\{\lambda:\mu([D(e_{(\lambda,\infty)}(|x|))>\lambda])\leq\lambda\}.
\end{align}

In the following, we will frequently use the following expression for the $F$-norm of projections (see \cite[Proposition 4.3.11]{BCLSZ}, \cite[page 262]{Yeadon 1973}) 
\begin{align}\label{p-norm}
\|p\|_{\mu,D}=\inf\limits_{0<\lambda<1}\max\{\lambda,\mu([D(p)>\lambda])\}, \,\, p\in P(\cM).
\end{align} 

For any $x\in \cM$, we have \cite[Proposition 4.3.18(ii)]{BCLSZ}
\begin{align}\label{lsmm}
\|x\|_{\mu,D}\le \|x\|_\cM.
\end{align}

We end this section with the following useful property of $\norm{\cdot}_{\mu,D}$. 
\begin{prop}\label{alpha-0} Let $\alpha: =\|\mathbf{1}\|_{\mu,D}$ and  $x\in LS(\cM)$. Then
    \begin{align*}
    \|x\|_{\mu,D}=\|x_\alpha\|_{\mu,D} \le \alpha,
    \end{align*}
where $x_\alpha=|x|e_{(0,\alpha)}(|x|)+\alpha e_{[\alpha,\infty)}(|x|)$.
\end{prop}

\begin{proof}
    Let
    \[
    y=|x|e_{(0,1]}(|x|)+e_{(1,\infty)}(|x|),
    \]
    then
    \begin{align*}
        e_{(\lambda,\infty)}(|x|)=e_{(\lambda,\infty)}(y),\,\,\text{for } 0<\lambda<1.
    \end{align*}
    Hence, by (\ref{second-form}),
    \begin{align*}
    \|x\|_{\mu,D}=\|y\|_{\mu,D}.
    \end{align*}
    Observe that
    \begin{align*}
        0\leq y\le  \mathbf{1}.
    \end{align*}
    Therefore, by the monotonicity of $F$-norm $\left\|\cdot \right\|_{\mu,D}$ on the positive part of $LS(\cM)$ \cite[Proposition 4.3.13 (\romannumeral3)]{BCLSZ}, we obtain
    \[
    \|y\|_{\mu,D}\le \|\mathbf{1}\|_{\mu,D}=\alpha,
    \]
    and consequently $\|x\|_{\mu,D}\le \alpha$.

Since $x_\alpha \le |x|$, we have $\|x_\alpha\|_{\mu,D}\le \|x\|_{\mu,D}$\cite[Proposition 4.3.13]{BCLSZ}.

For $\lambda\ge\alpha$, we have 
    \begin{align*}
    \max\{\lambda,\mu([D(e_{(\lambda,\infty)}(x_\alpha))>\lambda])\}\ge \alpha\ge \|x\|_{\mu,D}.
    \end{align*}
For $0<\lambda<\alpha$ we have
$e_{(\lambda, \infty)}(x_\alpha)=e_{(\lambda, \infty)}(|x|),$ and hence,
    \begin{align*}
    \max\{\lambda,\mu([D(e_{(\lambda,\infty)}(x_\alpha))>\lambda])\}=\max\{\lambda,\mu([D(e_{(\lambda,\infty)}(|x|))>\lambda])\} \stackrel{\eqref{second-form}}{\ge} \|x\|_{\mu,D}.
    \end{align*}
    Consequently, by definition, we have 
    \begin{align*}
    \|x_\alpha\|_{\mu,D} & \ge \|x\|_{\mu,D},
    \end{align*}
   which  completes the proof.
\end{proof}

\section{Isometries are Jordan $^*$-isomorphisms multiplied by an unitary}
The main result of this section is the following theorem. 

\begin{theorem}\label{isometry} Let $\mathcal{M}$ and $\mathcal{N}$ be von Neumann algebras and let  $\Phi:LS(\cM) \to LS(\cN)$ be  a surjective isometry.
Then  $\Phi$ is in the form \eqref{gen-form}.
\end{theorem}

Theorem~\ref{isometry} shows that the structure of isometries on $LS(\cM)$ parallels Kadison’s description for
$C^*$-algebras, despite the presence of unbounded operators.

We begin by establishing a fundamental asymptotic property of the
$F$-norm $\norm{\cdot}_{\mu,D}$. The key observation is that the asymptotic behavior of the
$F$-norm provides a criterion for distinguishing between bounded and unbounded operators.

\begin{lemma}\label{boundedcase}
 Let $\cM$ be a von Neumann algebra and let  $x\in \cM$. Then
\[
\lim_{\lambda\downarrow 0}\frac{\|\lambda x\|_{\mu,D}}{\lambda} = \|x\|_\cM.
\]
\end{lemma}

\begin{proof} The idea is to compare the scaling behavior of the $F$-norm with the operator norm using spectral projections of
$|x|$.

By \eqref{lsmm}, we have $\|\lambda x\|_{\mu,D}\le \lambda\|x\|_\cM$ for all $\lambda>0$.
Thus
\begin{align*}
\limsup_{\lambda\to 0}\frac{\|\lambda x\|_{\mu,D}}{\lambda} \le \|x\|_\cM.
\end{align*}

For the lower bound, fix $r\in(0,c)$, where $c:=\|x\|_\cM$. Then the spectral projection $p=e_{(c-r,\infty)}(|x|)$ is nonzero, and hence, $D(p)>0$ on a set of positive measure. Since  $\mu([D(p)>t])\uparrow \mu([D(p)>0])$ as $t\downarrow 0$, we can choose $\delta>0$ such that
\begin{align}\label{muDp}
\mu([D(p)>\delta]) > \gamma/2,
\end{align}
where $\gamma = \mu([D(p)>0])>0$. Take $\lambda$ small enough so that 
\begin{align}\label{c-ral}
\lambda(c-r) < \min\{\delta, \gamma/2\}.
\end{align}

We claim that $\|\lambda x\|_{\mu,D} \ge  \lambda(c-r)$. 
Assume by contradiction  that $\|\lambda x\|_{\mu,D} < \lambda(c-r)$.
By the definition of the $F$-norm in~\eqref{second-form} there exists $s>0$ with
\[
\max\{\lambda s,\; \mu([D(e_{(s,\infty)}(|x|)) > \lambda s])\} < \lambda(c-r).
\]
Hence, $\lambda s < \lambda(c-r)$ and
\begin{align}\label{D(e)}
\mu([D(e_{(s,\infty)}(|x|)) > \lambda s]) < \lambda(c-r)\stackrel{\eqref{c-ral}}{<}\gamma/2.
\end{align}
The first inequality gives $s < c-r$, and so, the spectral projection
$e_{(s,\infty)}(|x|)$ dominates $e_{(c-r,\infty)}(|x|)=p$.  Consequently,
$D(e_{(s,\infty)}(|x|)) \ge D(p)$, and therefore, by $\lambda s \le \lambda(c-r) \stackrel{\eqref{c-ral}}{<} \delta$, we have 
\[
[D(e_{(s,\infty)}(|x|)) > \lambda s] \supseteq [D(p) > \lambda s] \supseteq [D(p) > \delta].
\]   Hence
\[
\mu([D(e_{(s,\infty)}(|x|)) > \lambda s]) \ge \mu([D(p) > \delta]) \stackrel{\eqref{muDp}}{>} \gamma/2.
\]
However, the last inequality contradicts \eqref{D(e)}. From this contradiction we obtain that $\|\lambda x\|_{\mu,D} \ge \lambda(c-r)$ for all small $\lambda$ satisfying~\eqref{c-ral}.  Letting $r\to0$ gives $\liminf\limits_{\lambda\to0}\frac{\|\lambda x\|_{\mu,D}}{\lambda}\ge c$.  Combined with the upper limit, the limit exists and is equal to $\|x\|_\cM$. The proof is complete.
\end{proof}

\begin{lemma}\label{unboundedcase}
Let $\cM$ be a von Neumann algebra and let $x\in LS(\cM)\setminus \cM$. Then
\[
\lim_{\lambda\downarrow 0}\frac{\|\lambda x\|_{\mu,D}}{\lambda} = \infty.
\]
\end{lemma}

\begin{proof}
Let $x\in LS(\cM)\setminus \cM$. Since $x\notin \cM,$ it follows that  $e_{(n,\infty)}(|x|)\neq0$ for all $n\ge1$. Thus, for any $n\ge1$ there exists a number $k_n>n$ such that
$e_{(n,k_n)}(|x|)\neq 0$. Then
$$
|x|\ge |x|e_{(n,k_n)}(|x|)
$$
for all $n\ge1$. Using monotonicity of the $F$-norm $\|\cdot\|_{\mu,D}$, we have
\begin{align*}
\frac{\|\lambda x\|_{\mu,D}}{\lambda} & =\frac{\|\lambda |x|\|_{\mu,D}}{\lambda}\ge
\frac{\|\lambda |x|e_{(n,k_n)}(|x|)\|_{\mu,D}}{\lambda}.
\end{align*}
Taking into account $|x|e_{(n,k_n)}(|x|)\in \cM$ and Lemma~\ref{boundedcase}, we have
\begin{align*}
\lim_{\lambda\downarrow 0}\frac{\|\lambda x\|_{\mu,D}}{\lambda} &  \ge \lim_{\lambda\downarrow 0}
\frac{\|\lambda |x|e_{(n,k_n)}(|x|)\|_{\mu,D}}{\lambda}=\||x|e_{(n,k_n)}(|x|)\|_{\cM}>n
\end{align*}
for all $n\ge1$.
Since the inequality holds for every $n\ge1$, letting
$n\to \infty$ implies  $\lim\limits_{\lambda\downarrow 0} \frac{\|\lambda x\|_{\mu,D}}{\lambda} = \infty$.
\end{proof}

\begin{prop}\label{main:prop}
Let $\cM$ and $\cN$ be von Neumann algebras and let $x\in \cM$ and $y\in LS(\cN)$. If
\[
\|\lambda x\|_{\mu_{\cM},D_{\cM}}=\|\lambda y\|_{\mu_\cN,D_{\cN}}
\]
for all $\lambda>0$, then $y\in \cN$ and $\|x\|_\cM=\|y\|_\cN$.
\end{prop}

\begin{proof} Since $x\in\cM$, Lemma \ref{boundedcase} gives
\[
\lim_{\lambda\downarrow 0}
\frac{\|\lambda x\|_{\mu_{\cM},D_{\cM}}}{\lambda}
=
\|x\|_{\cM}
<\infty.
\]
By assumption, we have 
\[
\|\lambda x\|_{\mu_{\cM},D_{\cM}}
=
\|\lambda y\|_{\mu_{\cN},D_{\cN}},
\quad \forall \lambda>0.
\]
Hence
\[
\lim_{\lambda\downarrow 0}
\frac{\|\lambda y\|_{\mu_{\cN},D_{\cN}}}{\lambda}
=
\|x\|_{\cM}
<\infty.
\]
By Lemma \ref{unboundedcase}, we have  $y\in\cN$.
Applying Lemma \ref{boundedcase} again, we obtain 
\[
\|y\|_{\cN}
=
\lim_{\lambda\downarrow 0}
\frac{\|\lambda y\|_{\mu_{\cN},D_{\cN}}}{\lambda}
=
\|x\|_{\cM},
\]
completing the proof.
\end{proof}

\begin{proof}[Proof of Theorem~\ref{isometry}] Let $\Phi:LS(\cM)\to LS(\cN)$ be an isometry and let $x\in \cM$.
Then, 
\[
\|\lambda x\|_{\mu_{\cM},D_{\cM}} =\|\lambda \Phi(x)\|_{\mu_{\cN},D_{\cN}}
\]
for all $\lambda>0$. By Proposition~\ref{main:prop}, we obtain that $\Phi(x)\in \cN$ and $\|x\|_\cM=\|\Phi(x)\|_\cN$. So, the restriction $\Phi|_{\cM}$ is an operator-norm isometry from
$\cM$ onto $\cN$. By Kadison's Theorem \cite[Theorem 7]{K51} (see also~\eqref{descriptionofKadison}), it is in the form
\[
\Phi(x)=wJ(x),\, x\in \cM,
\]
where $w$ is a unitary element in $\cN$ and $J:\cM \to \cN$ is a Jordan $^*$-isomorphism.
Then $J$ is also an $\left\|\cdot\right\|_{\mu_{\cM},D_{\cM}}$-$\left\|\cdot\right\|_{\mu_{\cN},D_{\cN}}$ isometry on $\cM$, in particular, is continuous in the local measure topology\footnote{Note that $^*$-isomorphisms between von Neumann algebras are known to be continuous with respect to the local measure topology\cite[Theorem 4.3.32]{BCLSZ}.}.
Since $\cM$ is dense in $LS(\cM)$ in the local measure topology \cite[Corollary 4.3.15]{BCLSZ}, $J$  can be extended uniquely as a Jordan $^*$-isomorphism
from $LS(\cM)$ onto $LS(\cN)$, which we still denote as $J$ and satisfies \eqref{gen-form}, completing the proof.
\end{proof}

\section{Equivalence between pairs of dimension functions and probability measures}
\label{s4}
Recall that the local measure $F$-norm, defined by
\begin{align*}
\|x\|_{\mu, D} = \inf_{0<\lambda<1}\max\left\{\lambda,\; \mu\bigl([D(e_{(\lambda,\infty)}(|x|))>\lambda]\bigr)\right\}
\end{align*}
depends on a dimension function $D$ and a probability measure $\mu$.

In this section, we show that the $F$-norm $\left\|\cdot\right\|_{\mu,D}$ is uniquely determined by the dimension function $D$ and the probability measure $\mu$ in each of the following cases: 
when the von Neumann algebra $\mathcal M$ is  of type II$_\infty$ or type III;
it is of type II$_1$ with an atomless center or it is a II$_1$-factor.
In contrast, the situation changes substantially in the presence of type $\mathrm{II}_1$ components or for type~I algebras with  atomic centers of dimension at least two. In these cases, the above uniqueness property no longer holds.

Let $\cM$ and $\cN$ be von Neumann algebras  and let $(\mu_\cM, D_\cM)$ and $(\mu_\cN, D_\cN)$ be the corresponding pairs consisting of a dimension function and a probability measure.

Let $Z(\cM)\cong L_{\infty}(\Omega,\Sigma,\mu_\cM)$ and $Z(\cN)\cong L_{\infty}(\Omega',\Sigma',\mu_\cN)$. Suppose that $\Phi:(LS(\cM),\left\|\cdot\right\|_{\mu_{\cM},D_{\cM}})\to(LS(\cN),\left\|\cdot\right\|_{\mu_{\cN},D_{\cN}})$ is a surjective isometry. By Theorem~\ref{isometry}, it is in the form $\Phi=w J.$
Then $J$ maps $Z(\cM)$ onto $Z(\cN)$ and $J|_{Z(\cM)}$ is an $^*$-isomorphism from $Z(\cM)$ onto $Z(\cN)$. Thus
$$
Z(\cM)\cong L_{\infty}(\Omega,\Sigma,\mu_\cM)\cong L_{\infty}(\Omega',\Sigma',\mu_\cN).
$$ 
Taking into account that a Jordan $^*$-isomorphism $J:\cM\to\cN$ is normal \cite[page 111]{Jordan isomorphism}, it follows that the mapping
\[
D_{\cN}\circ J\big|_{P(\cM)}:P(\cM) \to L^+(\Omega', \Sigma', \mu_\cN)
\]
is a dimension function on $\cM$.

\begin{lemma}\label{two-} Let $\mu_\cM, D_\cM, \mu_\cN, D_\cN, J$ be as above. Then  the pairs $\left(\mu_\cM, D_\cM\right)$ and 
$\left(\mu_\cN, D_\cN\circ J|_{P(\cM)}\right)$ are equivalent.  
\end{lemma}

\begin{proof}
Since $J(\cdot)=w^*\Phi(\cdot)$ is also an isometry, we have 
\[
\|x\|_{\mu_\cM, D_\cM}=\|J(x)\|_{\mu_\cN, D_\cN}
\]
for all $x\in LS(\cM)$.  Therefore, it suffices to show that
\begin{align}\label{muN}
\|x\|_{\mu_\cN, D_\cN\circ J|_{P(\cM)}}=\|J(x)\|_{\mu_\cN, D_\cN}
\end{align}
for all $x\in LS(\cM)$. Since $\cM$ is dense in $LS(\cM)$ in the local measure topology \cite[Corollary 4.3.15]{BCLSZ}, it suffices to consider  $x\in \cM$.

Let us first consider the case when $x\in \cM$ is a positive elementary operator, that is, 
\begin{align}\label{lp_n}
x=\lambda_1 p_1+\ldots +\lambda_n p_n,
\end{align}
where $\lambda_1, \ldots, \lambda_n>0$ and $p_1, \ldots, p_n$ are mutually orthogonal projections.
Since $J$ is a Jordan $^*$-isomorphism, we have
$$
J(e_{(\lambda, \infty)}(x))=J(\sum\limits_{\lambda_i>\lambda} p_i)=e_{(\lambda, \infty)}(J(x))
$$ for all $\lambda>0$.
Thus, 
\begin{align*}
(D_\cN\circ J)(e_{(\lambda, \infty)}(x))=D_\cN(J(e_{(\lambda, \infty)}(x)))=D_\cN(e_{(\lambda, \infty)}(J(x))),    
\end{align*}
and hence,
\begin{align*}
\mu_\cN\left([(D_\cN\circ J)(e_{(\lambda, \infty)}(x))>\lambda]\right)=\mu_\cN\left([D_\cN(e_{(\lambda, \infty)}(J(x)))>\lambda]\right)    
\end{align*}
for all $\lambda>0$. Thus, by \eqref{second-form}, the last equality implies \eqref{muN} for all positive elements of the form \eqref{lp_n}. 
Since the set of all elements of the form \eqref{lp_n} is dense in the local measure topology in $\cM_+=\left\{x\in \cM: x \ge 0\right\}$ and
$J$ is an isometry, it follows that  \eqref{second-form} holds for all $x\in \cM_+$, and hence for all positive elements from $LS(\cM)$ (see \cite[Corollary 4.3.15]{BCLSZ}).

Now,  take an arbitrary $x\in LS(\cM)$. 
We begin by considering the special case where $J$ is either a $^*$-isomorphism or a $^*$-anti-isomorphism.

Assume  $J$ is an $^*$-isomorphism. We have
\begin{align*}
|J(x)|^2=J(x)^*J(x)=J(x^*x)=J(|x|^2)=J(|x|)^2,
\end{align*}
and hence, 
\begin{align}\label{J(x)}
|J(x)|=J(|x|).
\end{align}
Taking into account \cite[Proposition 4.3.13 (\romannumeral1)]{BCLSZ}, we have
\begin{align*}
\|x\|_{\mu_\cN, D_\cN\circ J|_{P(\cM)}} 
&  = \||x|\|_{\mu_\cN, D_\cN\circ J|_{P(\cM)}}
=\|J(|x|)\|_{\mu_\cN, D_\cN}\\
& \stackrel{\eqref{J(x)}}{=} \||J(x)|\|_{\mu_\cN, D_\cN}
=\|J(x)\|_{\mu_\cN, D_\cN}.
\end{align*}

Now assume  $J$ is an $^*$-anti-isomorphism. Following the same reasoning as in the proof of \eqref{J(x)}, we have 
\begin{align}\label{J*}
|J(x)|=J(|x^*|),
\end{align}
and
hence, we have
\begin{align*}
\|x\|_{\mu_\cN, D_\cN\circ J|_{P(\cM)}} &  = \||x^*|\|_{\mu_\cN, D_\cN\circ J|_{P(\cM)}}=\|J(|x^*|)\|_{\mu_\cN, D_\cN}\\
& \stackrel{\eqref{J*}}{=} \||J(x)|\|_{\mu_\cN, D_\cN}=\|J(x)\|_{\mu_\cN, D_\cN}.
\end{align*}

 The following property will be needed in the sequel. Let $z\in \cM$ be a central projection and let $x\in LS(\cM)$. Then
\begin{align}\label{x+}
\|x\|_{\mu_{\cN},D_{\cN}\circ J|_{P(\cM)}}=\||zx|+|(\mathbf{1}-z)x^*|\|_{\mu_{\cN},D_{\cN}\circ J|_{P(\cM)}},    
\end{align}
which follows from the following relations
\begin{align*}
e_{(\lambda, \infty)}(|zx|+|(\mathbf{1}-z)x^*|) & = ze_{(\lambda, \infty)}(|x|)+(\mathbf{1}-z)e_{(\lambda, \infty)}(|x^*|) \\
 & \sim ze_{(\lambda, \infty)}(|x|)+(\mathbf{1}-z)e_{(\lambda, \infty)}(|x|)=e_{(\lambda, \infty)}(|x|),
\end{align*}
here we have used that $e_{(\lambda, \infty)}(|x^*|)\sim e_{(\lambda, \infty)}(|x|)$ (see \cite[Proposition 2.1.7]{BCLSZ}) for all $\lambda>0$.

Next, we consider a general case.

By \cite[Proposition 3.1]{BHS24}, there exists a central projection $z'\in\cN$ such that $J(\cdot)z'$ is a $^*$-isomorphism and $J(\cdot)(\mathbf{1}-z')$ is a $^*$-anti-isomorphism. 
Since $J$ is bijective, there exists a central projection $z\in\cM$ such that $J(z)=z'$.
By \eqref{J(x)}, \eqref{J*}, we have
\begin{align}\label{JxJ*}
J(|zx|)=|J(zx)|,\,\, J(|(\mathbf{1}-z)x^*|)=|J((\mathbf{1}-z)x)|.
\end{align}
Using that Jordan $^*$-isomorphism maps central orthogonal summands to central orthogonal summands, and the above cases, we obtain 
\begin{eqnarray*}
\|x\|_{\mu_\cN, D_\cN\circ J|_{P(\cM)}} &  \stackrel{\eqref{x+}}{=} &\||zx|+|(\mathbf{1}-z)x^*|\|_{\mu_\cN, D_\cN\circ J|_{P(\cM)}}\\
& =&\|J(|zx|+|(\mathbf{1}-z)x^*|)\|_{\mu_\cN, D_\cN}\\
&=&\| J(|zx|)+J(|(\mathbf{1}-z)x^*|) \|_{\mu_\cN, D_\cN} \\
& \stackrel{\eqref{JxJ*}}{=}&\| |J(zx)|+|J((\mathbf{1}-z)x)| \|_{\mu_\cN, D_\cN}\\
&  =&\| |J(x)| \|_{\mu_\cN, D_\cN}=\| J(x) \|_{\mu_\cN, D_\cN},
\end{eqnarray*}
completing the proof.
\end{proof}

Now we are in a position to present the proof of Theorem~\ref{local-isometry}.

\begin{proof}[Proof of Theorem~\ref{local-isometry}] It suffices to show implication (1)$\Rightarrow$(2), because the converse implication is trivial. By Theorem~\ref{isometry}, $\Phi$ is in the form~\eqref{gen-form}, that is,
\[
\Phi(x)=wJ(x),\, x\in LS(\cM).
\]
Since $J(\cdot)=w^* \Phi(\cdot)$ is also an isometry, by Lemma~\ref{two-}, the pairs $(\mu_\cM, D_\cM)$ and $\left(\mu_\mathcal{N}, D_\mathcal{N}\circ J|_{P(\cM)}\right)$ are equivalent, completing the proof.  
\end{proof}

Let $(\mu_1, D_1)$ and $(\mu_2,D_2)$ be equivalent pairs associated with a von Neumann algebra $\cM$, that is,  
$(\Omega_i, \Sigma_i, \mu_i)$, $i=1,2$ are two probability measure spaces such that 
$Z(\cM)\cong L_\infty(\Omega_i, \Sigma_i, \mu_i)$, $i=1,2$ and  
$D_i: P(\cM)\to L^+(\Omega_i, \Sigma_i, \mu_i)$, $i=1,2$  are   two dimension functions on $P(\cM)$ such that 
\[
\left\|\cdot\right\|_{\mu_1,D_1}=\left\|\cdot\right\|_{\mu_2,D_2}.
\]

Let $\psi:L_{\infty}(\Omega_1,\Sigma_1,\mu_1)\to L_{\infty}(\Omega_2,\Sigma_2,\mu_2)$ be a $^*$-isomorphism. Define a probability measure $\mu_2\circ\psi$ on $(\Omega_1,\Sigma_1)$ by
\[
(\mu_2\circ\psi)(A)
   =
\int_{\Omega_2}\psi(\chi_A)(\omega)\,d\mu_2(\omega),
\qquad A\in\Sigma_1 .
\]
By construction, $\psi:L_\infty(\Omega_1,\Sigma_1,\mu_2\circ\psi)\to
L_\infty(\Omega_2,\Sigma_2,\mu_2)$
is a measure-preserving $^*$-isomorphism.

Define
\[
\widetilde{\psi}:
L^{+}(\Omega_1,\Sigma_1,\mu_2\circ\psi)
\to
L^{+}(\Omega_2,\Sigma_2,\mu_2)
\]
by
\[
\widetilde{\psi}(f)
=
\sup_{n\in\mathbb N}\psi(\min\{f, n\}).
\qquad
f\in L^{+}(\Omega_1,\Sigma_1,\mu_2\circ\psi),
\]
Then $\widetilde{\psi}$ extends $\psi$ from the positive cone of
$L_\infty(\Omega_1,\Sigma_1,\mu_2\circ\psi)$ to an order isomorphism
\[
L^{+}(\Omega_1,\Sigma_1,\mu_2\circ\psi)
\to
L^{+}(\Omega_2,\Sigma_2,\mu_2).
\]
Moreover, its inverse is obtained analogously from $\psi^{-1}$.

Since $\widetilde{\psi}$ is an order isomorphism, the mapping
\[
p\in P(\mathcal M)
\longmapsto
(\widetilde{\psi}^{-1}\circ D_2)(p)
\in
L^{+}(\Omega_1,\Sigma_1,\mu_2\circ\psi)
\]
defines a dimension function. Furthermore, since
\[
\psi:L_\infty(\Omega_1,\Sigma_1,\mu_2\circ\psi)
\to
L_\infty(\Omega_2,\Sigma_2,\mu_2)
\]
is measure-preserving, it follows immediately from \eqref{second-form} that the pairs $(\mu_2,D_2)$ and $\bigl(\mu_2\circ\psi,\widetilde{\psi}^{-1}\circ D_2\bigr)$
are equivalent. Consequently, $(\mu_1,D_1)$ and $\bigl(\mu_2\circ\psi,\widetilde{\psi}^{-1}\circ D_2\bigr)$
are also equivalent.

Therefore, the study of equivalence of pairs of probability measures and dimension functions reduced to the case when both measures defined on same 
domain.

We start with the following result.

\begin{lemma}\label{delta+}
Let $\cM$ be a von Neumann algebra, and let $p\in\mathcal M$ be a nonzero projection such that
\[
D(p)(\omega) \ge \mu([D(p)>0])
\]
for almost all $\omega\in[D(p)>0]$.
Then
\begin{align}\label{pmu}
\|p\|_{\mu,D}
=
\mu([D(p)>0]).
\end{align}
\end{lemma}

\begin{proof}
By assumptions,
\[
D(p)(\omega) \ge\mu([D(p)>0])=:\alpha
\quad \text{on } [D(p)>0],
\]
and $D(p)$ vanishes outside $[D(p)>0]$. 
Hence, $[D(p)>\lambda]=[D(p)>0]$ for all $0<\lambda<\alpha$.

If $0<\lambda<\alpha$, then
\[
\max\{\lambda,\mu([D(p)>\lambda])\}
=\max\{\lambda,\mu([D(p)>0])\}=
\max\{\lambda,\alpha\}
=
\alpha.
\]

If $\alpha\le \lambda< 1$, then
\[
\max\{\lambda,\mu([D(p)>\lambda])\}\ge \lambda
\ge\alpha.
\]

Since $\alpha=\mu([D(p)>0])$, we conclude that
$$
\|p\|_{\mu,D}\stackrel{(\ref{p-norm})}{=}\inf_{0<\lambda<1}\max\{\lambda,\mu([D(p)>\lambda])\}=\mu([D(p)>0]),
$$
which completes the proof.
\end{proof}

\begin{lemma}\label{properly central}Let $\cM$ be a properly infinite von Neumann algebra and let $z\in\cM$ be  a central projection.
Then
\begin{align}\label{zmu}
\|z\|_{\mu,D} =\mu([D(z)>0]).
\end{align}
In particular, if $(\mu, D)$ and $(\mu',D')$ are equivalent, then $\mu=\mu'$.
\end{lemma}

\begin{proof}
Let  $z$ be a nonzero  central projection. Since $\cM$ is properly infinite, it follows that $z$ is properly infinite. By property~\eqref{D-finite} of the dimension function, we have
\[
D(z)(\omega)=+\infty>\mu([D(z)>0]) \quad \text{on } [D(z)>0],
\]
and $D(z)$ vanishes outside of $[D(z)>0]$.  Hence, $z$ satisfies the condition of Lemma~\ref{delta+}. Therefore,
\begin{align*}
\|z\|_{\mu,D} & =
\mu([D(z)>0]).
\end{align*}

Assume  $(D, \mu)$ and $(D',\mu')$ are equivalent. For $A\in \Sigma$, we have 

$$
D(z_A)=D(z_A\mathbf{1})\stackrel{\eqref{D-finite5}}{=}\chi_A D(\mathbf{1}),
$$ 
and therefore 
\begin{align}\label{mu(A)}
[D(z_A)>0]=A.    
\end{align}

Hence, 
\begin{align*}
\mu(A) \stackrel{\eqref{mu(A)}}{=}\mu([D(z_A)>0]) \stackrel{\eqref{zmu}}{=}  \|z_A\|_{\mu,D}=\|z_A\|_{\mu', D'}\stackrel{\eqref{zmu}}{=}\mu'([D'(z_A)>0])=\mu'(A),
\end{align*}
completing the proof.
\end{proof}

Now we are able to prove the following result.

\begin{prop}\label{III}
Let $\cM$ be of type III. If $(\mu, D)$ and $(\mu',D')$ are equivalent, then they coincide.
\end{prop}

\begin{proof} By Lemma~\ref{properly central}, we have $\mu=\mu'$.
By \cite[Corollary 1.11.4]{BCLSZ}, dimension functions on type III algebras are unique, and hence,  $D=D'$.
\end{proof}

\subsection{The type  $II$ case}

Let $\mathcal{M}$ be a semifinite von Neumann algebra  with center $Z(\mathcal{M})\cong L_\infty(\Omega,\Sigma,\mu)$.

Let $T$ be a faithful normal semifinite extended center-valued trace on the algebra $\cM$\cite[Theorem 2.34]{Tak1}.
By \cite[Theorem 1.11.13]{BCLSZ}, every dimension function $D$ on $P(\mathcal{M})$ has the form
\begin{align}\label{Dpc}
D(p)=c\,T(p),\,\, p\in P(\cM),
\end{align}
where $c:\Omega\to(0,\infty)$ is a strictly positive measurable function.

\begin{lemma}\label{equiv}
Let $\cM$ be of type II. If $(\mu,D)$ and $(\mu,D')$ are equivalent, then they coincide.
\end{lemma}

\begin{proof} 
Assume that $D\neq D'$. Then there exists a finite projection $e\in \cM$ such that $D(e)\neq D'(e)$. Without loss of generality, we may assume that the subset $[D'(e)>D(e)]$ has a positive $\mu$-measure. 
 By \eqref{dim-fc} there exists  a faithful normal semifinite extended center-valued trace $T$ on $\cM$ such that
$D=T|_{P(\cM)}$.
By \eqref{Dpc} we can find  a strictly positive measurable function
$c':\Omega\to(0,\infty)$ such that $D'=c'T|_{P(\cM)}=c'D$. 
Since $[D'(e)>T(e)=D(e)]$ has a positive $\mu$-measure, it follows that $[c'>1]$ also has a positive $\mu$-measure. We can find  a measurable subset $A\subset [D'(e)>T(e)]$ of positive measure and $\lambda_0>0$ such that $A\subset [T(e)>\lambda_0]$. 
Since 
\[
T(z_Ae)= \chi_A T(e)\ge \lambda_0\chi_A, 
\]
it follows from   \cite[Theorem 5.6.2]{SS08} that  there exists a subprojection $p$ of $z_A e$ such that
\begin{align}\label{Tpl}
T(p)=\lambda_0\chi_A,
\end{align}
in particular, the central support $c(p)$ of $p$ is $c(p)=z_A$. 
We can find a measurable subset $B\subset A$ with a positive measure $r$ and $\varepsilon>0$ such that 
\begin{align}\label{Bsub}
B\subset [D'(p)=c'T(p)>T(p)+\varepsilon].
\end{align}
Let $n$ be a positive integer such that $\frac{\lambda_0+\varepsilon}{2^n}<r=\mu(B)$. 
Since $\cM$ is of type II, using the Halving Lemma \cite[Lemma 6.5.6]{KR2}, we can find 
mutually orthogonal and mutually equivalent projections $q_1, \ldots, q_{2^n}$ such that 
$\sum\limits_{i=1}^{2^n}q_i=z_B p$.
We have
\[
\lambda_0\chi_B = \lambda_0\chi_A \chi_B \stackrel{\eqref{Tpl}}{=} \chi_B T(z_A p)=T(z_B p)=\sum\limits_{i=1}^{2^n}T(q_i)=\sum\limits_{i=1}^{2^n}T(q_1)=2^n T(q),
\]
where $q=q_1$. Thus 
\[
2^nD'(q)=c' 2^nT(q) = c'T(z_Bp)=c'T(p)\chi_B\stackrel{\eqref{Bsub}}{\ge} T(p)\chi_B+\varepsilon\chi_B=(\lambda_0+\varepsilon)\chi_B.
\]

Let us compute $\|q\|_{\mu,D}$ and $\|q\|_{\mu,D'}$.

For $\frac{\lambda_0}{2^n}\le \lambda <1$, the set $[D(q)>\lambda]$ is empty, because $D(q)=\frac{\lambda_0}{2^n}\chi_B$. Thus
\[
\inf\limits_{\frac{\lambda_0}{2^n}\le \lambda <1}
\max\left\{\lambda, \mu\left([D(q)>\lambda]\right)\right\}
=
\inf\limits_{\frac{\lambda_0}{2^n}\le \lambda <1}
\max\left\{\lambda, 0\right\}
=
\frac{\lambda_0}{2^n}.
\]

For $0<\lambda<\frac{\lambda_0}{2^n}$, we have $[D(q)>\lambda]=B$, and hence
$$
\inf\limits_{0<\lambda<\frac{\lambda_0}{2^n}}
\max\left\{\lambda, \mu\left([D(q)>\lambda]\right)\right\}
=
\inf\limits_{0<\lambda<\frac{\lambda_0}{2^n}}
\max\{\lambda, r\}
=r.
$$
Hence, by \eqref{p-norm}, we have
\[
\|q\|_{\mu,D}= \frac{\lambda_0}{2^n}.
\]

By a similar argument, we obtain
\[
\inf\limits_{\frac{\lambda_0+\varepsilon}{2^n}\le \lambda <1}
\max\left\{\lambda, \mu\left([D'(q)>\lambda]\right)\right\}
\ge \frac{\lambda_0+\varepsilon}{2^n}
\]
and
\[
\inf\limits_{0<\lambda<\frac{\lambda_0+\varepsilon}{2^n}}
\max\left\{\lambda, \mu\left([D'(q)>\lambda]\right)\right\}
=
\inf\limits_{0<\lambda<\frac{\lambda_0+\varepsilon}{2^n}}
\max\left\{\lambda, r\right\}
>
\frac{\lambda_0+\varepsilon}{2^n}.
\]
Thus,
\[
\|q\|_{\mu,D'}\ge  \frac{\lambda_0+\varepsilon}{2^n}.
\]
Therefore, $\|q\|_{\mu,D}\neq \|q\|_{\mu, D'}$. From this contradiction we obtain that $D'=D$.
\end{proof}

Lemmas~\ref{properly central}, \ref{equiv} imply the following result.

\begin{prop}\label{II-infty}
Let $\cM$ be of type II$_\infty$. If $(\mu,D)$ and $(\mu',D')$ are equivalent, then they coincide.
\end{prop}

\subsection{The case when the center is  atomless}

\begin{lem}\label{mu=mu}
Let $\cM$ be a von Neumann algebra with atomless center $Z(\cM)$. If $(\mu,D)$ and $(\mu',D')$ are equivalent, then $\mu=\mu'$.
\end{lem}

\begin{proof} Recall that $Z(\cM)\cong L_\infty(\Omega, \Sigma, \mu)\cong L_\infty(\Omega, \Sigma, \mu')$. Let $A$ be a measurable subset of $\Omega$.

We will prove the statement by three steps.

{\it Step 1}. Suppose
\begin{align}\label{DzAd}
D(z_A)\ge \mu([D(z_A)>0])\chi_A,\,\,\, D'(z_A)\ge \mu'([D'(z_A)>0])\chi_A.
\end{align}
Using \eqref{pmu}, we have
\begin{align*}
\mu(A) & = \mu([D(z_A)>0])\stackrel{\eqref{pmu}}{=}\|z_A\|_{\mu,D}=\|z_A\|_{\mu',D'} \stackrel{\eqref{pmu}}{=}\mu'([D'(z_A)>0])=\mu'(A).
\end{align*}

{\it Step 2}. Suppose
\begin{align}\label{DzA}
D(z_A)\ge \mu([D(z_A)>0])\chi_A.
\end{align}
Since $Z(\cM)$ is atomless, we can find partition $\{A_i\}_{i\ge1}$ of $A$ such that\footnote{Let $h$ be a positive measurable function on an atomless measure space $\Omega$ with a measure $\nu$. Then there exists a partition $\{A_i\}_{i\ge1}$ of ${\rm supp}(h)$ and  $r_i>0$ such that $\chi_{A_i}h\ge r_i \chi_{A_i}$ for all $i\ge1$. Since $\Omega$ is atomless, for each $i$ we can find a partition $\{A_{i,j}\}_{1\le j\le k_i}$ of $A_i$ such that $\nu(A_{i,j})<r_i$ for all $1\le j\le k_i$. Then 
\begin{align}\label{homega} \chi_{A_{i,j}}h\ge \nu(A_{i,j})\chi_{A_{i,j}}\end{align} for all $i,j$.}
\begin{align}\label{dza}
D'(z_{A_i})& = D'(z_{A})\chi_{A_i}\stackrel{\eqref{homega}}{\ge} \mu'(A_i)\chi_{A_i}=\mu'([D'(z_{A_i})>0])\chi_{A_i},\,\, i\ge 1.
\end{align}
For every $i\ge1$, we have
\begin{align}\label{dzaD}
D(z_{A_i})=\chi_{A_i}D(z_A)\stackrel{\eqref{DzA}}{\ge} \chi_{A_i}\mu([D(z_{A})>0])\ge \chi_{A_i}\mu([D(z_{A_i})>0]).
\end{align}
Inequalities \eqref{dza} and  \eqref{dzaD} show that each $A_i$ satisfies \eqref{DzAd}. Hence, by Step 1, we have $\mu(A_i)=\mu'(A_i)$ for all $i\ge1$. Thus
\[
\mu(A)=\sum\limits_{i\ge1} \mu(A_i)=\sum\limits_{i\ge1} \mu'(A_i)=\mu'(A).
\]

{\it Step 3.}
Now consider an arbitrary $A$.
Since $Z(\cM)$ is atomless, again as in Step 2, we can find partition $\{A_i\}_{i\ge 1}$ of $A$ such that
\[
D(z_{A_i})\ge \mu([D(z_{A_i})>0])\chi_{A_i},\,\, i\ge 1.
\]
By Step 2, we have $\mu(A_i)=\mu'(A_i)$. Thus
\[
\mu(A)=\sum\limits_{i\ge1} \mu(A_i)=\sum\limits_{i\ge1} \mu'(A_i)=\mu'(A),
\]
completing the proof.
\end{proof}

Lemmas~\ref{equiv} and  \ref{mu=mu} imply the following result.

\begin{prop}\label{II-1}
Let $\cM$ be a von Neumann algebra of type II$_1$ with the atomless center. If $(\mu,D)$ and $(\mu',D')$ are equivalent, then they coincide.
\end{prop}

Now we are in a position to present the proofs of Theorem~\ref{con-equiv}.

\begin{proof}[Proof of Theorem~\ref{con-equiv}] Part (a) follows from Propositions~\ref{III} and~\ref{II-infty}, including the factor cases.

For part (b), when the center of $\cM$ is atomless, the result follows from Proposition~\ref{II-1}, while the factor case follows from Lemma~\ref{equiv}.
\end{proof}

\begin{remark}
As one can see from the proofs of Propositions~\ref{III}, \ref{II-infty}, and~\ref{II-1}, only the quantities $\|p\|_{\mu,D}$ for $p\in P(\mathcal{M})$ are compared there. Consequently, the arguments given in those propositions are sufficient  to prove Corollary~\ref{uniq}.
\end{remark}

\section{The case for 
type $I$ factors or algebras with 
atomless center 
}

In this section, we present the proof of  Theorem~\ref{case of type I}.

Let $L_0(\Omega)=L_0(\Omega, \Sigma, \mu)$ be  the space of equivalence classes of measurable complex functions on an atomless probability measure space $\Omega$.
Then a dimension function $D$ on $P(L_\infty(\Omega))\equiv \left\{\chi_A: A\in \Sigma\right\}$ is defined as
\[
D(\chi_A) = c\chi_A, A \in \Omega,
\]
where $c:\Omega\to(0,\infty)$ is a strictly positive measurable function. Then
\[
D(e_{(t, \infty)}(|x|))=D(\chi_{[|x|>t]})=c\chi_{[|x|>t]},
\]
and hence,
\[
[D(e_{(t, \infty)}(|x|))>t]=[|x|>t]\cap [c>t].
\]

For $x\in L_0(\Omega)$,
set
\[
m_{x,c}(t):
=
\mu([|x|>t]\cap[c>t]).
\]
Then by (\ref{first-form}), the $F$-norm $\left\|\cdot\right\|_{\mu,D}$ can be represented as
\begin{align}\label{function-norm}
\|x\|_{\mu, D} =
\inf\limits_{0<t<1}\{t:m_{x,c}(t)\le t\}.
\end{align}
On the other hand,  \eqref{second-form} can be rewritten as follows 
\begin{align}\label{function-norm2}
    \|x\|_{\mu,D}=\inf\limits_{0<t<1}\max\{t,m_{x,c}(t)\}.
\end{align}

Define the critical value
\[
\alpha_{\mu,D}
=
\|\mathbf{1}\|_{\mu,D}
=
\inf\limits_{0<t<1}\{t:\mu([c>t])\le t\}.
\]
We have 
\begin{align}\label{muct}
\mu([c>t])\le t\quad\text{for  } \alpha_{\mu,D}\le t<1,
\end{align}
and
\begin{align}\label{muct1}
\mu([c>t])>t \quad \text{for  } 0<t<\alpha_{\mu,D}.
\end{align}

\begin{prop}\label{cc'} Let $\cM=L_\infty(\Omega, \Sigma, \mu)$ be an atomless abelian von Neumann algebra and let $c$, $c':\Omega\to(0,\infty)$ be strictly positive measurable functions.
The following are equivalent:
\begin{enumerate}
\item $(\mu,D)$ and $(\mu',D')$ are equivalent.
\item $\mu=\mu'$, $\|\mathbf{1}\|_{\mu,D}=\|\mathbf{1}\|_{\mu',D'}$ and
      $\min\{c, \|\mathbf{1}\|_{\mu,D}\}=\min\{c', \|\mathbf{1}\|_{\mu',D'}\}$.
\end{enumerate}
\end{prop}

\begin{proof}
(2)$\Rightarrow$(1).
Let $\alpha=\alpha_{\mu,D}=\alpha_{\mu,D'}$ and assume
\[
\min\{c,\alpha\}=\min\{c',\alpha\}.
\]
If $t<\alpha$, then
\begin{align*}
[c>t]=[\min\{c,\alpha\}>t]=[\min\{c',\alpha\}>t]=[c'>t],
\end{align*}
then we have
\[
m_{x,c}(t)=m_{x,c'}(t).
\]
If $t\ge\alpha$, then
\[
m_{x,c}(t)\le\mu([c>t])\stackrel{\eqref{muct}}{\le} t,
\quad
m_{x,c'}(t)\le\mu([c'>t])\stackrel{\eqref{muct}}{\le} t.
\]
Hence, by (\ref{function-norm}), the defining infima coincide and $\|x\|_{\mu, D}=\|x\|_{\mu', D'}$.

(1)$\Rightarrow$(2).
Assume $\left\|\cdot \right\|_{\mu,D}= \left\|\cdot \right\|_{\mu',D'}$. By Lemma~\ref{mu=mu}, we have $\mu=\mu'$.

Taking $x\equiv1$ gives
\[
\alpha:=\alpha_{\mu, D}=\alpha_{\mu, D'}.
\]

Suppose that 
\[
[\min\{c,\alpha\}\neq \min\{c',\alpha\}]
\]
has a positive measure. Without loss of generality, we can assume that the set 
$[\min\{c,\alpha\}>\min\{c',\alpha\}]$ has a positive measure. Since
$$
[\min\{c,\alpha\}>\min\{c',\alpha\}]=[\alpha > c>c']\sqcup[c\ge \alpha >c'],
$$
and $$
\bigcup_{n\geq 1}\left[\alpha>c>c-\frac{1}{n}>c'\right]=
\left[\alpha>c>c'\right],$$
$$
\bigcup_{n\geq 1}\left[c\ge \alpha>\alpha-\frac{1}{n}>c'\right]=
\left[c\ge \alpha >c'\right],$$
it follows that there exist numbers $0<r'<r<\alpha$ such that
\[
[c>r>r'>c']
\]
has a positive measure. Since the measure space is atomless, we can find a measurable subset
$A\subset [c>r>r'>c']$ such that $\epsilon=\mu(A)>0$, where $\epsilon<r-r'$.
By \eqref{muct1} we have $\mu([c>r])>r$.
Again we can find a measurable subset $B\subset [c>r]\setminus A$ such that $\mu(B)=r-\epsilon$.
Note that
\[
\mu(A\sqcup B)=\mu(A)+\mu(B)=\epsilon+r-\epsilon=r.
\]
Set $x=r\chi_{A\sqcup B}$.
By construction, we have
$$
[|x|>t]\cap [c>t] =
\begin{cases}
\emptyset, & \text{if $t\ge r$},\\
A\sqcup B, & \text{if $t<r$}.
		 \end{cases}
$$
Taking into account $\mu(A\sqcup B)=r$, we obtain
\begin{align*}
\|x\|_{\mu,D} & \stackrel{(\ref{function-norm2})}{=}\inf\limits_{0<t<1}\max\{t, \mu([|x|>t]\cap [c>t])\} \\
& =\inf\limits_{0<t<r}\max\{t, \mu(A\sqcup B)\}=r.
\end{align*}
Since $A\subset [r'>c']$,  for every $t\in (r',r)$, we have  
$$
[|x|>t]\cap [c'>t]\subset [|x|>t]\cap [c'>r'] \subset B.
$$
Taking into account $\mu(B)=r-\epsilon>r'$, we obtain
\begin{align*}
\|x\|_{\mu,D} & \stackrel{(\ref{function-norm2})}{=}\inf\limits_{0<t<1}\max\{t, \mu([|x|>t]\cap [c'>t])\}\\
& \le \inf\limits_{r'<t<r}\max\{t, \mu([|x|>t]\cap [c'>t])\}\le \inf\limits_{r'<t<r}\max\{t, \mu(B)\}=r-\epsilon.
\end{align*}
This contradicts equality of norms.
Thus $\mu([\min\{c,\alpha\}\neq \min\{c',\alpha\}])=0$, and therefore
$\min\{c,\alpha\}=\min\{c',\alpha\}$.
The proof is complete.\end{proof}

Let $\cM$ be a type I von Neumann algebra with an atomless center $Z(\cM)\equiv L_{\infty}(\Omega,\Sigma,\mu)$, where $(\Omega, \Sigma, \mu)$ is an atomless probability measure space.
Fix  a faithful normal semifinite extended center-valued trace $T$ on $\cM$ such that $T(p)=1$ on $\Omega$, where $p$ is a faithful ($c(p)=\mathbf{1}$) abelian projection in $\cM$. Then any dimension function $D$ on $P(\cM)$ is uniquely defined by the value  $D(p)=c$.

\begin{lemma}\label{reduce-to-finite}
Let $\cM$ be a  von Neumann algebra of type I  with the  center $Z(\cM)\cong L_{\infty}(\Omega,\Sigma,\mu)$.
Then there exists a finite projection $q\in P(\cM)$ such that
\[
\|\mathbf{1}\|_{\mu,D}=\|q\|_{\mu,D}.
\]
\end{lemma}

\begin{proof} Let $\Omega = A\sqcup B$ be a partition of $\Omega$ into measurable parts such that $z_A\cM$ is finite and $z_B\cM$ is infinite. 
If the projection $z_B$ is zero, then $\mathbf{1}=z_A$ is finite and there is nothing to prove.
There, 
we may assume that   $z_B$ is nonzero.

Take a partition $\{B_i\}_{i\ge1}$ of $B$ and a sequence $\{n_i\}_{i\ge1}$ such that
\begin{align}\label{B_i}
\sum\limits_{i\ge1}\frac{1}{n_i}\chi_{B_i}\le c \chi_B.
\end{align}
Let $\{q_i\}_{i\ge1}$ be a sequence of mutually orthogonal faithful abelian projections in $z_B\cM$ (see \cite[page 299 the proof of Theorem 1.27]{Tak1}.
Set
\[
q=z_A+\sum\limits_{i\ge1}\sum\limits_{j=1}^{n_i}q_jz_{B_i}\in P_{fin}(\cM).
\]
Since $q_i$'s are equivalent\cite[Proposition 6.4.6]{KR2}, it follows that $D(q_j)=D(p)$ and  
\begin{align*}
\chi_{B_i}D(q) & = \sum\limits_{j=1}^{n_i}D(q_j)\chi_{B_i}=\sum\limits_{j=1}^{n_i}D(p)\chi_{B_i}=
n_i c \chi_{B_i}\stackrel{\eqref{B_i}}{\ge}\chi_{B_i}.
\end{align*}
Therefore,
\begin{align}\label{D(q)1}
\chi_B D(q) & \ge \chi_B.
\end{align}

Let us show that 
\begin{align}\label{q+z_E}
\|q\|_{\mu,D}=\|\mathbf{1}\|_{\mu,D}.
\end{align}

For $0<\lambda<1$,  we have 
\begin{align*}
[D(q)>\lambda] & = [D(z_A q +z_Bq)>\lambda]=[\chi_A D(\mathbf{1})+\chi_B D(q)>\lambda]\\
& =[\chi_A D(\mathbf{1})>\lambda]\sqcup [\chi_B D(q)>\lambda] \stackrel{(\ref{D(q)1})}{=} [\chi_A D(\mathbf{1})>\lambda] \sqcup B\\
& = [\chi_A D(\mathbf{1})>\lambda]\sqcup [\chi_B D(\mathbf{1})>\lambda]= [D(\mathbf{1})>\lambda].
\end{align*}
Hence, taking into account \eqref{p-norm}, we obtain  \eqref{q+z_E}. The proof is complete.
\end{proof}

Let $D'$ be another dimension function on $P(\cM)$ with $D'(p)=c'$.
Set
\begin{align*}
c_{\min} =\min\{c, \|\mathbf{1}\|_{\mu,D}\},\,\,\, D_{\min}=c_{\min}T,
\end{align*}
and
\begin{align*}
c'_{\min} =\min\{c', \|\mathbf{1}\|_{\mu,D'}\},\,\,\, D'_{\min}=c'_{\min}T.
\end{align*}

\begin{prop}\label{I-case} Let $\cM$ be a type I von Neumann algebra with an atomless center $Z(\cM)$. The pairs $(\mu,D)$ and $(\mu,D')$ are equivalent if and only if  $D_{\min}=D'_{\min}$.
\end{prop}

\begin{proof} $(\Leftarrow)$. Let us first  show that
\[
\left\|\cdot\right\|_{\mu,D}=\left\|\cdot\right\|_{\mu,D_{\min}} .
\]

Take $x\in LS(\cM)$ and  set
$$
x_\alpha=|x|e_{(0,\alpha)}(|x|)+\alpha e_{[\alpha,\infty)}(|x|),
$$
where $\alpha=\|\mathbf{1}\|_{\mu,D}$.

Let $e=e_{(\lambda, \infty)}(x_\alpha)$, $\lambda>0$. By
\cite[Proposition 6.4.6(ii)]{KR2}, we have
\[
e\succeq c(e)p,
\]
where $c(e)$ denotes the central support of $e$. Since the dimension
function $D$ is monotone, it follows that
\begin{equation}\label{supp(e)}
D(e)\ge \chi_E D(p)
=\chi_E c,
\end{equation}
where $E\in \Sigma$ such that $c(e)=z_E$ (see~\eqref{zA}).
The same inequality holds with $D_{\min}$ in place of $D$.

Set
\[
A:=[c\ge \alpha]\cap E, \qquad B:=[c<\alpha]\cap E.
\]

Let $0<\lambda<\alpha$. Then, by \eqref{supp(e)},  we obtain
\begin{align*}
A\cap [D(e)>\lambda] & \stackrel{\eqref{supp(e)}}{=}A =A\cap[D_{\min}(e)>\lambda].
\end{align*}
Since $\chi_B c=\chi_B c_{\min}$, it follows that $\chi_B D=\chi_B D_{\min}$, and hence,
\begin{align*}
B\cap [D(e)>\lambda] &  =B\cap[D_{\min}(e)>\lambda].
\end{align*}
Thus, 
\begin{align*}
[D(e)>\lambda]
&=(A\cap [D(e)>\lambda])\sqcup (B\cap [D(e)>\lambda]) \\
&=
(A\cap [D_{\min}(e)>\lambda])
\sqcup (B\cap [D_{\min}(e)>\lambda]) \\
&=[D_{\min}(e)>\lambda].
\end{align*}
Thus,
\begin{align*}
\|x_\alpha\|_{\mu,D}
&=\inf_{0<\lambda<\alpha}
\max\Bigl\{
\lambda,\,
\mu\bigl([D(e_{(\lambda,\infty)}(x_{\alpha}))>\lambda]\bigr)
\Bigr\} \\
&=
\inf_{0<\lambda<\alpha}
\max\Bigl\{
\lambda,\,
\mu\bigl([D_{\min}(e_{(\lambda,\infty)}(x_{\alpha}))>\lambda]\bigr)
\Bigr\} =\|x_\alpha\|_{\mu,D_{\min}}.
\end{align*}
By Proposition~\ref{alpha-0}, we have
\[
\|x\|_{\mu,D}=\|x_\alpha\|_{\mu,D}=\|x_\alpha\|_{\mu,D_{\min}}=\|x\|_{\mu,D_{\min}}
\]

Now assume that $D_{\min}=D'_{\min}$. Then
\[
\left\|\cdot\right\|_{\mu,D}=\left\|\cdot\right\|_{\mu,D_{\min}}=\left\|\cdot\right\|_{\mu,D'_{\min}}=\left\|\cdot\right\|_{\mu,D'}.
\]

 $(\Rightarrow)$. Assume that $\left\|\cdot\right\|_{\mu, D}=\left\|\cdot\right\|_{\mu, D'}$. 
By Lemma~\ref{reduce-to-finite} there exists a finite projection $q\in P(\cM)$ such that
$\|q\|_{\mu, D}=\|\mathbf{1}\|_{\mu,D}$.
Replacing $\cM$ to $q\cM q$, we can assume that $\cM$ is a finite von Neumann algebra of type I.

Assume $c_{\min}\neq c'_{\min}$. Without loss of generality,
we may  assume that  the set $[c'_{\min}>c_{\min}]$ has a positive measure. Since
$$\bigcup_{n\geq 1}\left[\alpha-\frac{1}{n}\geq c_{\min}'>c_{\min}>\frac{1}{n}\right]=\left[c_{\min}'>c_{\min}\right],$$
it follows that there exists $m\in\mathbb{N}$ such that the set $\left[\alpha-\frac{1}{m}\geq c_{\min}'>c_{\min}>\frac{1}{m}\right]$ has a positive measure. Taking into account $Z(\cM)$ which is atomless, we can find a measurable set 
\begin{align}\label{A}
 A\subset \left[\alpha-\frac{1}{m}\geq c_{\min}'>c_{\min}>\frac{1}{m}\right]
\end{align}
with the measure  $\mu(A)=\frac{1}{m}$.

We have 
$$
D(z_{A})\geq D(z_{A}p)=\chi_Ac\stackrel{(\ref{A})}{=}\chi_{A}c_{\min}\stackrel{(\ref{A})}{>}\frac{1}{m}\chi_{A}.
$$
Thus, 
\begin{align*}
    \mu(A)=\frac{1}{m}<D(z_A)(\omega)
\end{align*}
for almost all $\omega\in A$. By Lemma \ref{delta+}, we have 
\begin{align*}
    \|z_A\|_{\mu,D}=\mu(A)=\frac{1}{m}.
\end{align*}

Set
\[
e:=z_{\Omega\setminus A}+z_Ap.
\]
Consider the abelian von Neumann algebra $Z(\cM)e$ with the dimension function $D|_{P(Z(\cM)e)}$. 
Observing that $\|z_A(\mathbf{1}-p)\|_{\mu,D}\stackrel{\mbox{\tiny  \cite[Prop. 4.3.13 (\romannumeral2)]{BCLSZ}}}{\leq}\|z_A\|_{\mu,D}=\frac{1}{m}$   and $e=\mathbf{1}-z_A(\mathbf{1}-p)$, we have
\begin{align}\label{q}
    \|e\|_{\mu,D}\geq\|\mathbf{1}\|_{\mu,D}-\|z_A(\mathbf{1}-p)\|_{\mu,D}\geq\alpha-\frac{1}{m}.
\end{align}

Note that 
\begin{align}\label{D(q)=}
    D(e)=D(z_{\Omega\setminus A}+z_Ap)=\chi_{\Omega\setminus A}D(\mathbf{1})+\chi_Ac,
\end{align}
and 
\begin{align}\label{D_(q)}
D'(e)=\chi_{\Omega\backslash A}D' (\mathbf{1})+\chi_Ac'.
\end{align}

Applying  Proposition \ref{cc'} to the abelian von Neumann algebra $Z(\cM)e$, we have
\begin{align}\label{=}
    \min\{D(e),\|e\|_{\mu,D}\}=\min\{D'(e),\|e\|_{\mu,D'}\}
\end{align}
In addition,
\begin{align*}
\chi_A \|e\|_{\mu,D}\stackrel{\eqref{q}}{\ge} \chi_A \left(\alpha -\frac{1}{m}\right)\stackrel{\eqref{A}}{\ge} \chi_A c\stackrel{\eqref{D(q)=}}{=} \chi_A D(e).      
\end{align*}
Similarly, using \eqref{D_(q)}, we have
\begin{align*}
\chi_A \|e\|_{\mu,D'}\ge \chi_A D'(e).      
\end{align*}
Thus,
\begin{align*}
    \chi_A\min\{D(e),\|e\|_{\mu,D}\}=\chi_Ac=\chi_Ac_{\min},
\end{align*}
\begin{align*}
\chi_A\min\{D'(e),\|e\|_{\mu,D'}\}=\chi_Ac'=\chi_Ac_{\min}'.
\end{align*} 
By (\ref{=}), we have 
$$\chi_Ac_{\min}=\chi_Ac'_{\min},$$, which is in contradiction to the choice of the set $A$ (see \eqref{A}).
The proof is complete.
\end{proof}
\begin{remark}
It should be noted that ``if'' part of the above Proposition~\ref{I-case} holds for a type I von Neumann algebra with an arbitrary center.
\end{remark}

Now we consider the case of type I-factors.

For a type I factor $\mathcal{M}=B(H)$ we have $LS(\mathcal{M})=\mathcal{M}$ \cite[Theorem 2.5.7]{BCLSZ}, and its center 
$Z(\mathcal{M})=\mathbb{C}\mathbf{1}$. Hence,  the underlying measure space is reduced to a single point, that is, $\Omega=\{\omega\}$ with $\mu(\{\omega\})=1$. Consequently, any dimension function on $P(B(H))$ is a real-valued function and is uniquely determined by its value on a minimal projection.
More precisely, if  $p$ is a minimal projection and suppose $D(p)=\delta>0$, then  for every projection $q\in P(B(H))$,  we have
\[
D(q)=\delta\,\mathrm{rank}(q),
\]
where ${\rm rank}(q)=\dim q(H)$ is the rank of projection $q$.

\begin{lemma}\label{I-factor}
Let $p$ be a minimal projection in $B(H)$. Then 
\begin{align*}
\|p\|_{\mu,D}=\min\{D(p), 1\}.
\end{align*}
In particular, $(\mu,D)$ and $(\mu,D')$ are equivalent if and only if $\min\{D(p),1\}=\min\{D'(p),1\}$.
\end{lemma}

\begin{proof}
If $D(p)\ge 1$, then \[
\mu([D(p)>\lambda])=1
\]
for all $0<\lambda<1$. Thus
\[
\|p\|_{\mu,D}\stackrel{\eqref{p-norm}}{=}\inf\limits_{0<\lambda <1}\max\left\{ \lambda, \mu\left(\left[D(p) > \lambda\right]\right) \right\}=
\inf\limits_{0<\lambda <1}\max\left\{ \lambda, 1\right\}=1.
\]

If $0<D(p)<1$, then
$$
\max\left\{ \lambda, \mu\left(\left[D(p) > \lambda\right]\right) \right\}=
\begin{cases}
1, & \text{if $\lambda<D(p)$},\\
\lambda, & \text{if $\lambda\ge D(p)$}.
		 \end{cases}
$$
Thus, 
\[
\|p\|_{\mu,D}\stackrel{\eqref{p-norm}}{=}\inf\limits_{0<\lambda<1}\max\left\{ \lambda, \mu\left(\left[D(p) > \lambda\right]\right) \right\}
=D(p).
\]
Therefore, 
\[
\|p\|_{\mu,D}
=
\min\{D(p),1\},
\]
completing the proof.
\end{proof}

\begin{proof}[Proof of Theorem~\ref{case of type I}] The proof  follows directly from Propositions~\ref{mu=mu},~\ref{I-case} and Lemma~\ref{I-factor}.
\end{proof}

\section{The case for von Neumann algebra having atomic centers}\label{s6}

We show that the assumptions in Theorems~\ref{con-equiv}(b) and~\ref{case of type I} -- namely, that the algebra is of type II$_1$, or  of type I with an atomless center, or  a factor -- are essential. Specifically, we will show below that for a von Neumann algebra of type II$_1$ or type I whose center contains at least two atoms, both theorems fail.

\subsection{Example $1$}
Let $\cN$ be a von Neumann factor with a faithful normal tracial state $\tau$.
Set
\[
\mathcal{M}=\bigoplus_{i=1}^n \cN.
\]
The center of $\cM$ is $Z(\mathcal{M})\cong L_\infty(\Omega,\Sigma,\mu)$ with $\Omega=\{1,\ldots, n\}$, $\mu(\{i\})=m_i>0$, $i=1,\ldots, n$ and
$\sum\limits_{i=1}^n m_i=1$.

Define
\[
D(p)=(d_1 \tau(p_1),\ldots, d_n \tau(p_n)),\,\, p=(p_1, \ldots, p_n)\in P(\cM),
\]
where $d_i>0$, $i=1, \ldots, n$.
Suppose that 
\begin{align}\label{dm}
\max\{d_1, \ldots, d_n\}<\min\{m_1, \ldots, m_n\}.
\end{align}

Let us  compute $\|x\|_{\mu,D}$, where  $x\in LS(\cM)$.

Let $x = (x_1, \ldots, x_n)\in LS(\cM)$.    We have
\[
D(e_{(\lambda,\infty)}(|x|)) = \Bigl(d_1\tau\bigl(e_{(\lambda,\infty)}(|x_1|)\bigr),\ldots,
d_n\tau\bigl(e_{(\lambda,\infty)}(|x_n|)\bigr)\Bigr).
\]
Thus, 
$$
[D(e_{(\lambda,\infty)}(|x|)) > \lambda]=\left\{i: d_i\tau\bigl(e_{(\lambda,\infty)}(|x_i|)\bigr) > \lambda\right\}.
$$
Now, for any fixed $\lambda\in(0,1)$, the quantity $\max\{\lambda,\,
\mu([D(e_{(\lambda,\infty)}(|x|)) > \lambda])\}$ equals
\[
\max\Bigl\{\lambda,\;\mu\left(\left\{i: d_i\tau\bigl(e_{(\lambda,\infty)}(|x_i|)\bigr) > \lambda\right\}\right)\Bigr\}.
\]
Observe that $\tau(e_{(\lambda,\infty)}(|x_i|))$ is a nonincreasing function of $\lambda$ with values in $[0,1]$.  
Define
\begin{align}\label{lambda-i}
\lambda_i(x)= \sup_{0<\lambda<1}\left\{\tau(e_{(\lambda,\infty)}(|x_i|)) > \frac{\lambda}{d_i}\right\},\, i=1, \ldots,n.
\end{align}
Since  $\tau(e_{(\lambda,\infty)}(|x_i|))\le 1$, we have $\lambda_i(x)\le d_i$ for all $i$.

For $\lambda > \lambda_{\max}(x):=\max\{\lambda_1(x),\ldots, \lambda_n(x)\}$, the set $[D(e_{(\lambda,\infty)}(|x|)) > \lambda]$ is empty,  so the maximum of the set $\{\lambda,\;\mu(\{i:d_i\tau(e_{(\lambda,\infty)}(\abs{x_i}))>\lambda\})\}$ is $\lambda$ itself.  The infimum of   $\max \{\lambda,\;\mu(\{i:d_i\tau(e_{(\lambda,\infty)}(\abs{x_i}))>\lambda\})\}$
over $\lambda > \lambda_{\max}(x) $ is then $\max\{\lambda_1(x),\ldots, \lambda_n(x)\}$.

Let $\lambda <   \lambda_{\max}(x)=\max\{\lambda_1(x), \ldots, \lambda_n(x)\}$. Let   $i$ be such that $\lambda <\lambda_i(x)=\max\{\lambda_1(x), \ldots, \lambda_n(x)\}$, in particular, $i\in [D(e_{(\lambda,\infty)}(|x|))  > \lambda]$.
Thus, 
\begin{align*}
\mu\left([D(e_{(\lambda,\infty)}(|x|))  > \lambda]\right) & \ge m_i\ge \min\{m_1, \ldots, m_n\} \\
&>\max\{d_1, \ldots, d_n\}\ge d_i\ge \lambda_i(x)=\lambda_{\max}(x) >\lambda,
\end{align*}
and therefore, the infimum of   $\max \{\lambda,\;\mu(\{i:d_i\tau(e_{(\lambda,\infty)}(\abs{x_i}))>\lambda\})\}$
over $0<\lambda <\lambda_{\max}(x) $ is strictly larger than $\lambda_{\max}(x)$.

Hence, the infimum $\max \{\lambda,\;\mu(\{i:d_i\tau(e_{(\lambda,\infty)}(\abs{x_i}))>\lambda\})\}$
over over $0<\lambda<1$ is exactly $\max\{\lambda_1(x), \ldots, \lambda_n(x)\}$.
Thus, we obtain
\begin{align}\label{dd}
\left\|\bigoplus_{i=1}^n x_i\right\|_{\mu,D} = \max\{\lambda_1(x), \ldots, \lambda_n(x)\}.
\end{align}

Assume $n\ge 2$ and  take an arbitrary $\mu'=(m'_1,\ldots, m'_n)$ also satisfying~\eqref{dm}. By \eqref{dd}, we have
\[
\|x\|_{\mu,D}=\|x\|_{\mu', D}
\]
for all $x\in LS(\cM)$. This means that  the identity mapping
\[
\mathrm{Id}:(LS(\mathcal{M}),\left\|\cdot\right\|_{\mu,D}) \to (LS(\mathcal{M}),\left\|\cdot\right\|_{\mu',D})
\]
is an isometry, which shows that  an  isometry on $LS(\cM)$ may not preserve measures.

As a concrete example, 
let us take $\mathcal{N}=\mathbb{C}$. In this case every spectral projection
$e_{(\lambda,\infty)}(|x_i|)$ is either $0$ or $1$. Consequently, the quantity
$\lambda_i(x)$ defined in \eqref{lambda-i} reduces to
\[
\lambda_i(x) = \min\{d_i, |x_i|\}, \qquad x=(x_1,\ldots,x_n)\in \mathcal{M}=\mathbb{C}^n .
\]

Therefore, the $F$-norm given in \eqref{dd} can be rewritten as
\[
\|x\|_{\mu,D} = \max \bigl\{ \min\{d_1, |x_1|\}, \ldots, \min\{d_n, |x_n|\} \bigr\},
\qquad x=(x_1,\ldots,x_n)\in \mathcal{M}=\mathbb{C}^n .
\]

Thus, in the scalar case the $F$-norm is determined by the maximum of the truncated absolute values of the coordinates. In particular, it depends only on the vector
$D=(d_1,\ldots,d_n)$ and is independent of $\mu=(m_1, \ldots, m_n)$.

\subsection{Example $2$}
Let $\mathcal{M}=B(H)\oplus B(H)$, where $H$ is a Hilbert space with the dimension at least $2$.
The center of $\cM$ is $Z(\mathcal{M})\cong L_\infty(\Omega,\Sigma,\mu)$ with $\Omega=\{1,2\}$, $\mu(\{i\})=m_i$, $i=1,2$.

Define a measure $\mu$ and a dimension function $D$ by
\[
m_1=2/3,\, m_2=1/3,
\]
\[
D(p)=(d_1{\rm Tr}(p_1),\,d_2{\rm Tr}(p_2)),\qquad p=(p_1,p_2)\in P(\mathcal{M}),
\]
where 
\begin{align}\label{dd1}
d_1=1, \, \frac12<d_2<\frac23,
\end{align}
 and ${\rm Tr}$ is the standard trace on $B(H)$ (the usual trace, which takes values in $[0,\infty]$, and is integer-valued or infinity on projections).

Let $x=x_1\oplus x_2$, where $x_1,x_2\ne 0$. 
Set
\[
\alpha_i:=\|x_i\|_{B(H)}, \, i=1,2,
\]
where $\left\|\cdot\right\|_{B(H)}$ is the operator norm on $B(H)$.

Note that ${\rm Tr}(e_{(\alpha_2-\varepsilon, \infty)}(|x_2|))\ge 1$ for all $0<\varepsilon<\alpha_2$.
Set
\[
\gamma=
\begin{cases}
1, & {\rm if   }\,\,\, {\rm Tr}(e_{(\alpha_2-\varepsilon, \infty)}(|x_2|))\ge 2,\,\, \forall \varepsilon\in (0,\alpha_2)\\
0, & {\rm otherwise   }.
\end{cases}
\]
The $F$-norm can be computed as follows   
\begin{align}\label{dm1}
\|x\|_{\mu,D}=
\begin{cases}
\min\{\alpha_1, \alpha_2, 1\}, & {\rm if   }\,\,\, 2/3\le \alpha_1, \alpha_2,\, \gamma=1,\\
2/3, & {\rm if   }\,\,\, 2/3\le \alpha_1,  \alpha_2,\, \gamma=0,\\
\alpha_1, & {\rm if   }\,\,\, 1/3\le \alpha_1<2/3,\\
1/3, & {\rm if   }\,\,\, \alpha_1\le 1/3, \alpha_2\ge 1/3,\\
\max\{\alpha_1,\alpha_2\}, & {\rm if   }\,\,\, \alpha_1,\alpha_2\le 1/3.
\end{cases}
\end{align}
 
Below, we only present the   measures of $A(\lambda):=[D(e_{(\lambda, \infty)}(|x|))>\lambda]$, $0<\lambda<1,$ which is enough to obtain \eqref{dm1}. A straightforward computation shows that
\begin{enumerate}
\item $\mu(A(\lambda))=\begin{cases}
    1, & 0<\lambda<\min\{\alpha_1,\alpha_2,1\},\\
    \leq\frac{2}{3}, & \lambda\geq\min\{\alpha_1,\alpha_2,1\};
\end{cases}$
\item $\mu(A(\lambda))=\begin{cases}
    1, & 0<\lambda<d_2,\\
    \frac{2}{3}, & d_2\leq\lambda<\frac{2}{3},\\
    \leq\frac{2}{3}, & \lambda\geq\frac{2}{3};
\end{cases}$
\item $\mu(A(\lambda))=\begin{cases}
    \geq\frac{2}{3}, & 0<\lambda<\alpha_1,\\
    \leq\frac{1}{3}, & \lambda\geq\alpha_1;
\end{cases}$
\item $\mu(A(\lambda))=\begin{cases}
    \geq\frac{1}{3}, & 0<\lambda<\frac{1}{3},\\
    \leq\frac{1}{3}, & \lambda\geq\frac{1}{3};
\end{cases}$
\item $\mu(A(\lambda))=\begin{cases}
    \geq\frac{1}{3}, & 0<\lambda<\max\{\alpha_1,\alpha_2\},\\
    0, & \lambda\geq\max\{\alpha_1,\alpha_2\}.
\end{cases}$
\end{enumerate}

Now, 
taking arbitrary $D'=(d_1, d'_2)$ also satisfying~\eqref{dd1}, by \eqref{dm1}, we have
\[
\|x\|_{\mu,D}=\|x\|_{\mu, D'}
\]
for all $x\in LS(\mathcal{M})$. This means that  the identity mapping
\[
\mathrm{Id}:(LS(\mathcal{M}),\left\|\cdot\right\|_{\mu,D}) \to (LS(\mathcal{M}),\left\|\cdot\right\|_{\mu,D'})
\]
is an isometry, which shows that the  identity isometry does not preserve dimension functions.

\subsection{Example $3$}
Below,  we present an example showing that the identity mapping is an isometry but preserves neither the measure nor the dimension function.

Consider $\cM=\mathbb{C}^4$. Then $\cM\cong L_{\infty}(\Omega,\Sigma,\mu)$ with $\Omega=\{1,2,3,4\}$, $\mu(\{i\})=m_i$, $i=1,2,3,4$. Define a probability measure $\mu$ and a dimension function $D$ by
\begin{align}\label{measure2}
\frac{1}{32}\le m_1,\,m_2,\ m_1+m_2=\frac{1}{8},\,  m_3=\frac{1}{4},\ m_4=\frac{5}{8},
\end{align}
and
$$
D(p)=(d_1p_1,d_2p_2,d_3p_3,d_4p_4),\ p=(p_1,p_2,p_3,p_4),\ p_i\in\{0,1\},
$$
where 
\begin{align}\label{di2}
d_1=d_2=\frac{1}{32},\ \frac{1}{2}<d_3<\frac{5}{8},\ d_4=1.
\end{align}
Let $x=(x_1, x_2, x_3, x_4)\in \mathcal{M}$. Straightforward and tedious computations show that
\begin{align}\label{formula3}
\|x\|_{\mu,D}=
\begin{cases}
\min\{|x_4|, \frac{5}{8}\}, & {\rm if   }\,\,\, |x_4|\ge \frac{1}{4},\\
\frac{1}{4}, & {\rm if   }\,\,\, |x_4|<\frac{1}{4}\le |x_3|,\\
\max\{|x_3|,|x_4|\}, & {\rm if   }\,\,\, \frac{1}{32} \le \max\{|x_3|,|x_4|\}<\frac{1}{4},\\
\frac{1}{32} & {\rm if   }\,\, |x_3|,|x_4|<\frac{1}{32}\le \max\{|x_1|,|x_2|\},\\
\max\{|x_1|,|x_2|, |x_3|,|x_4|\}, & {\rm if   }\,\,\, |x_1|,|x_2|, |x_3|,|x_4|<\frac{1}{32}.
\end{cases}
\end{align}

Taking arbitrary $m'=\{m_1',m_2',m_3,m_4\}$ satisfying (\ref{measure2}) and $D'=\{d_1,d_2,d_3',d_4\}$ satisfying (\ref{di2}), then by formula (\ref{formula3}) we have
$$\|x\|_{\mu,D}=\|x\|_{\mu',D'}.$$
In the example above, the identity mapping
$${\rm Id}:(LS(\cM),\left\|\cdot\right\|_{\mu,D})\to (LS(\cM),\left\|\cdot\right\|_{\mu',D'})$$ defines an isometry, but does not preserve the measure or the dimension function.

\begin{remark}
We may  say that a measure captures the ``width'' of a von Neumann algebra, while a dimension function captures its ``length''. If the center is atomless, 
then the width can be partitioned into arbitrarily small pieces. For type II or type III von Neumann algebras, the length can be partitioned into arbitrarily small pieces. These properties allow us to obtain  that equivalent pairs of measures and dimension functions actually coincide.
\end{remark}
\begin{remark}
Let $\mathcal{M}$ be a von Neumann algebra with the center $Z(\mathcal{M}) \cong L_\infty(\Omega,\Sigma,\mu)$. In the present work, we focus on isometries for the case where $\mu$ is a probability measure. Note that a general finite measure case can be reduced to the probability measure case by scaling.

Assume $\mu$ is a strictly positive finite measure. For any $r>0$, define a new measure $\mu_r$ on $\Sigma$ and a new dimension function $D_r$ on $P(\mathcal{M})$ by
\[
\mu_r(A)=r\,\mu(A)\quad\text{for }A\in\Sigma,\qquad
D_r(p)=r\,D(p)\quad\text{for }p\in P(\mathcal{M}).
\]
If $\left\|\cdot\right\|_{\mu,D}$ denotes the $F$-norm defined by \eqref{LSM-norm}, then a direct computation yields
\begin{align}\label{finite-measure}
\left\| r x \right\|_{\mu_r,D_r}=r\,\left\|x\right\|_{\mu,D}\qquad\text{for all }x\in LS(\mathcal{M}).
\end{align}
Consequently, a finite-measure case (with total mass $\mu(\Omega)$) can be reduced to the probability measure case by choosing
$r=\frac{1}{\mu(\Omega)}$ and using the scaling relation
~\eqref{finite-measure}.
\end{remark}

\section*{Acknowledgements} The authors were supported by the NNSF of China (No.12031004, 12301160 and 12471134) and by Basic Research Program of Jiangsu (BK20251783). 

{\bf Conflict of interest:} On behalf of all authors, the corresponding author states
that there is no conflict of interest.

{\bf Data availability:} Not applicable.

\end{document}